\documentclass[11pt,a4paper]{article}

\usepackage[hyperfootnotes=false,pdfpage layout=SinglePage]{hyperref}
\usepackage{setspace}
\usepackage{amssymb}
\usepackage{amsmath}
\usepackage{amsthm}
\usepackage{latexsym}
\usepackage{amsfonts}
\usepackage{mathrsfs}

\newtheoremstyle{my_theoremstyle}
  {5mm}
  {1mm}
  {\it}
  {0pt}
  {\bfseries}
  {{\bf .} }
  {0mm}
  {}
\newtheoremstyle{my_definitionstyle}
  {5mm}
  {1mm}
  {\rm}
  {0pt}
  {\bfseries}
  {{\bf .} }
  {0mm}
  {}
\theoremstyle{my_theoremstyle}
\newtheorem{thm}{Theorem}[section]
\newtheorem{cor}[thm]{Corollary}
\newtheorem{lem}[thm]{Lemma}
\newtheorem{prop}[thm]{Proposition}
\theoremstyle{my_definitionstyle}
\newtheorem{defn}[thm]{Definition}
\newtheorem{rem}[thm]{Remark}

\renewenvironment{proof}{\vspace{-.3cm}\noindent{\it Proof. }}{\vspace{.2cm}\qed}

\renewenvironment{enumerate}
{
\begin{list}{}
{\setlength{\topsep}{-.3cm} \setlength{\parsep}{0cm} \setlength{\itemsep}{.1cm} \setlength{\leftmargin}{1.0cm} \setlength{\labelwidth}{1.0cm}}
}
{\end{list}}

\newcommand{\emails}[1]
{\renewcommand{\thefootnote}{}\footnotetext{#1}
 \renewcommand{\thefootnote}{\arabic{footnote}}}

\numberwithin{equation}{section}
\makeatletter
\renewcommand\section{\@startsection{section}{1}{\z@}%
                                  {-5ex \@plus -1ex \@minus -.2ex}%
                                  {2.3ex \@plus.2ex}%
                                  {\normalfont\large\bfseries}}
\renewcommand\subsection{\@startsection{subsection}{1}{\z@}%
                                  {-5ex \@plus -1ex \@minus -.2ex}%
                                  {0.1ex \@plus.2ex}%
                                  {\normalfont\large\it}}
\makeatother
\renewcommand{\emptyset}{\varnothing}

\newcommand{\R}{\ensuremath{\mathbb R}}
\newcommand{\C}{\ensuremath{\mathbb C}}
\newcommand{\N}{\ensuremath{\mathbb N}}

\newcommand{\gperp}{{[\perp]}}

\newcommand{\product}{[\cdot\,,\cdot]}

         \newcommand{\frakB}{\mathfrak B}

\newcommand{\calH}{\mathcal H}

\newcommand{\calK}{\mathcal K}         
\newcommand{\calL}{\mathcal L}         
\newcommand{\calM}{\mathcal M}

\newcommand{\calS}{\mathcal S}

\newcommand{\calX}{\mathcal X}         
\newcommand{\calY}{\mathcal Y}

\newcommand{\la}{\lambda}
\newcommand{\veps}{\varepsilon}

\renewcommand{\Im}{\operatorname{Im}}
\renewcommand{\Re}{\operatorname{Re}}

\renewcommand{\ker}{\operatorname{ker}}
\newcommand{\ran}{\operatorname{ran}}

\newcommand{\sap}{\sigma_{{ap}}}
\renewcommand{\sp}{\sigma_{+}}
\newcommand{\sm}{\sigma_{-}}
\newcommand{\spp}{\sigma_{++}}
\newcommand{\smm}{\sigma_{--}}

\newcommand{\Lra}{\Longrightarrow}

\newcommand{\Llra}{\Longleftrightarrow}

\newcommand{\downto}{\downarrow}

\newcommand{\ol}{\overline}

\newcommand{\ds}{\dotplus}
\newcommand{\wt}{\widetilde}

\newcommand{\cls}{\operatorname{c.l.s.}}

\begin{document}
\thispagestyle{empty}
\begin{center}
\begin{spacing}{1.5}
{\LARGE\bf Locally definite normal operators in Krein spaces}
\end{spacing}

\vspace{0.6cm}
{\Large Friedrich Philipp}

\vspace{.3cm}
{\it Institut f\"ur Mathematik, MA 6-4, Technische Universit\"at Berlin, Stra\ss e des 17.\ Juni 136, 10623 Berlin, Germany}
\end{center}

\vspace{.7cm}\hrule

\vspace{.3cm}
{\bf Abstract}

\vspace{.3cm}
We introduce the spectral points of two-sided positive type of bounded normal operators in Krein spaces. It is shown that a normal operator has a local spectral function on sets which are of two-sided positive type. In addition, we prove that the Riesz-Dunford spectral subspace corresponding to a spectral set which is only of positive type is uniformly positive. The restriction of the operator to this subspace is then normal in a Hilbert space.

\vspace{.3cm}

\vspace{.4cm}\hrule

\vspace{0.5cm}

\emails{{\it Email address: }{\tt fmphilipp@gmail.com} (F.\ Philipp)}

\setlength{\parskip}{3ex plus 0.5ex minus 0.2ex}

\section{Introduction}
In 1987 J.\ Wu proved in \cite{wj} that for {\it every} bounded linear operator $A$ in a Hilbert space $\calH$ there exists a Krein space $\calK$ and a normal operator $B$ in this Krein space such that $\calH\subset\calK$ and $B|\calH = A$. In other words: every bounded linear operator in a Hilbert space is a "part" of a normal operator in a Krein space. If the Hilbert space $\calH$ is finite-dimensional then the Krein space $\calK$ can even be chosen as $\calH$ itself (see \cite[Proposition 8.1.2]{grl}).

From this point of view it seems desirable to have a profound spectral theory for bounded normal operators in Krein spaces. But the literature on normal operators in Krein spaces is very limited at present and, in addition, in each of the existing contributions global assumptions on the space or the normal operator are imposed. In \cite{ls} the existence of a spectral function for a normal operator in a Pontryagin space was proved and a complete classification of normal operators in a $\Pi_1$-space was worked out. In \cite{x} it is stated without proof that there exists a functional calculus for normal operators in Pontryagin spaces. In \cite{es} the concept of definitizability was extended from selfadjoint operators to a class of normal operators in a Krein space the spectrum of which does not have interior points. For such operators the existence of a spectral function with singularities was proved. Another special class of normal operators with a maximal non-negative invariant subspace was considered in \cite{as}. The papers \cite{aj,b98,b99} and \cite{v} deal with bounded and compact perturbations of fundamentally reducible normal operators.

In contrast to the above-quoted papers very weak assumptions on the normal operator (more precisely, on its real and imaginary part) were imposed in \cite{pst} and the notion of the spectrum of positive and negative type for selfadjoint operators in Krein spaces from \cite{lamm,lmm} was extended to normal operators. It could be shown that a normal operator has a local spectral function on open subsets of $\C$ which are of positive or negative type. This result is known for arbitrary selfadjoint operators in Krein spaces (see \cite{lmm}). But due to the global assumptions the result from \cite{pst} is not a proper generalization of that in \cite{lmm}. However, it shows that the spectrum of positive and negative type is also meaningful for normal operators.

In the present paper we continue the study of the spectral points of positive and negative type for normal operators, but we do not impose any assumptions on the Krein space inner product or the global structure of the operator. As in \cite{pst} it is our main objective to tackle the question whether or when a spectral point $\la$ of positive type of the normal operator $N$ in a Krein space has a neighborhood on which there exists a local spectral function for $N$. We prove that for this it is necessary that $\ol\la$ is a spectral point of positive type of the Krein space adjoint $N^+$ of $N$. This motivates us to introduce the set $\spp(N)$ which consists of all $\la\in\C$ such that $\la\in\sp(N)$ and $\ol\la\in\sp(N^+)$ (here, $\sp(\cdot)$ denotes the spectrum of positive type) and call it the {\it spectrum of two-sided positive type} of $N$. And indeed, we are able to prove that a normal operator has a local spectral function on sets which are of two-sided positive type (see Theorem \ref{t:lsf_pt}). Since for a selfadjoint operator $A$ the sets $\spp(A)$ and $\sp(A)$ coincide, Theorem \ref{t:lsf_pt} is a generalization of the above-mentioned result from \cite{lmm}.

At this point and in light of the results in \cite{pst} the natural question arises whether the sets $\sp(N)$ and $\spp(N)$ coincide for all normal operators. It is proved in Theorem \ref{t:spectral_set} that a spectral set (in the Dunford-Schwartz sense) which is of positive type is in fact of two-sided positive type. This essentially improves a re\-sult from \cite{pst} and shows, in particular, that the part of the operator $N$ corresponding to the spectral set is a normal operator in a Hilbert space. But the question whether $\sp(N) = \spp(N)$ holds in general has to be left open.

The paper is organized as follows. In section \ref{s:spdt} we define the spectral points of positive type and two-sided positive type and prove some of their basic properties. In section \ref{s:lsf} it is shown that a normal operator has a local spectral function on sets of two-sided positive type which is the first main result of this paper. The second main result is proved in section \ref{s:spectral_set} and states that the spectral subspace corresponding to a spectral set of positive type is a Hilbert space with respect to the Krein space inner product. As a consequence of the two main results we prove that $\sp(N) = \spp(N)$ holds in Pontryagin spaces.

\section{Spectral points of definite and two-sided definite type}\label{s:spdt}
In this paper, the Banach algebra of all bounded linear operators mapping from a Banach space $X$ to a Banach space $Y$ will be denoted by $L(X,Y)$. As usual, we set $L(X) := L(X,X)$. The spectrum (resolvent set) of $T\in L(X)$ is denoted by $\sigma(T)$ ($\rho(T)$, respectively).

Recall that for $T\in L(X)$ the {\it approximate point spectrum} $\sap(T)$ of $T$ consists of those $\la\in\sigma(T)$ for which there exists a sequence $(x_n)\subset X$ with $\|x_n\| = 1$, $n\in\N$, and $(T - \la)x_n\to 0$ as $n\to\infty$. Such a sequence is called an {\it approximate eigensequence for $T - \la$}. The points from the set $\C\setminus\sap(T)$ are called {\it points of regular type} of $T$. Note that $\la\notin\sap(T)$ if and only if $T - \la$ is injective and $\ran(T - \la)$ is closed.
%

Throughout this paper, let $(\calH,\product)$ be a Krein space. For the basic properties of (operators in and between) Krein spaces we refer to the monographs \cite{ai} and \cite{b}. We fix a Hilbert space norm $\|\cdot\|$ on $\calH$ such that
$$
\big|[x,y]\big|\,\le\,\|x\|\,\|y\|\quad\text{for all }x,y\in\calH.
$$
Such a norm exists, and all such norms are equivalent, cf.\ \cite{ai,b}. By $T^+$ we denote the adjoint of an operator $T\in L(\calH)$ with respect to the inner product $\product$. The statements of the following lemma will be used frequently without reference, cf.\ \cite[Chapter 2, Theorems 1.11 and 1.16]{ai}.

\begin{lem}\label{l:sap_adj}
Let $T\in L(\calH)$. Then the following statements hold.
\begin{enumerate}
\item[{\rm (i)}]   $\la\in\sigma(T)\;\Llra\;\ol\la\in\sigma(T^+)$.
\item[{\rm (ii)}]  $\la\in\sigma(T)\setminus\sap(T)\;\Lra\;\ol\la\in\sigma_p(T^+)\subset\sap(T^+)$.
\item[{\rm (iii)}] If $\calL$ is a $T$-invariant subspace, then $\calL^\gperp$ is $T^+$-invariant.
\end{enumerate}
\end{lem}

Hereby, $\calL^\gperp$ denotes the {\it orthogonal companion} of $\calL$ with respect to the inner product $\product$:
$$
\calL^\gperp := \{x\in\calH : [x,\ell] = 0\,\text{ for all }\ell\in\calL\}.
$$
A closed subspace $\calL\subset\calH$ is called {\it uniformly positive} ({\it uniformly negative}) if there exists $\delta > 0$ such that $[x,x]\ge\delta\|x\|^2$ ($-[x,x]\ge\delta\|x\|^2$, respectively) holds for all $x\in\calL$. Equivalently, the inner product space $(\calL,\product)$ ($(\calL,-\product)$, respectively) is a Hilbert space. In this case, we have $\calH = \calL\,[\ds]\,\calL^\gperp$, where $[\ds]$ denotes the direct $\product$-orthogonal sum.

Let us recall the definition of a local spectral function (of positive type) for a bounded operator, cf.\ \cite{lmm}.

\begin{defn}\label{d:lsf_pt}
Let $\calS\subset\C$ be Borel-measurable. By $\frakB(\calS)$ we denote the system of Borel-measur\-able subsets of $\calS$ whose closure is also contained in $\calS$. A mapping $E$ from $\frakB(\calS)$ into the set of all bounded projections in $(\calH,\product)$ is called a {\em local spectral function} for the operator $T\in L(\calH)$ on $\calS$ if for all $\Delta,\Delta_1,\Delta_2,\ldots\in\frakB(\calS)$ the following conditions are satisfied:
\begin{enumerate}
\item[{\rm (S1)}] $E(\Delta_1\cap\Delta_2) = E(\Delta_1)E(\Delta_2)$.
\item[{\rm (S2)}] If $\Delta_1,\Delta_2,\ldots\in\frakB(\calS)$ are mutually disjoint and $\bigcup_{k=1}^\infty\,\Delta_k\in\frakB(S)$, then
$$
E\left(\bigcup_{k=1}^\infty\,\Delta_k\right) = \sum_{k=1}^\infty\,E(\Delta_k),
$$
where the sum converges in the strong operator topology.
\item[{\rm (S3)}] $TB=BT\;\;\Longrightarrow\;\;E(\Delta)B=BE(\Delta)$ \;for every $B\in L(\calH)$.
\item[{\rm (S4)}] $\sigma(T|E(\Delta)\calH)\,\subset\,\ol{\sigma(T)\cap\Delta}$.
\item[{\rm (S5)}] $\sigma(T|(I-E(\Delta))\calH)\,\subset\,\ol{\sigma(T)\setminus\Delta}$.
\end{enumerate}
A local spectral function $E$ for $T$ on $\calS$ is said to be {\em of positive {\rm (}negative{\rm )} type} if for all $\Delta\in\mathfrak B(\calS)$
\begin{enumerate}
\item[{\rm (S6)}] $E(\Delta)\calH$ is uniformly positive {\rm (}uniformly negative, respectively{\rm )}.
\end{enumerate}
\end{defn}

For the rest of this paper let $N$ be a normal operator in $(\calH,\product)$, i.e.\ $N$ commutes with its adjoint,
$$
NN^+ = N^+N.
$$
The spectral points of positive and negative type defined below were first introduced in \cite{lamm} for bounded selfadjoint operators.

\begin{defn}\label{d:sp}
A point $\la\in\sap(N)$ is called a {\em spectral point of positive {\rm (}negative{\rm )} type} of the normal operator $N$ if for every sequence $(x_n)\subset\calH$ with $\|x_n\| = 1$, $n\in\N$, and $(N - \la)x_n\to 0$ as $n\to\infty$ we have
$$
\liminf_{n\to\infty}\,[x_n,x_n] > 0\quad\left(\limsup_{n\to\infty}\,[x_n,x_n] < 0,\;\text{respectively}\!\right)\!.
$$
The set of all spectral points of positive (negative) type of $N$ is denoted by $\sp(N)$ ($\sm(N)$, respectively). A set $\Delta\subset\C$ is said to be of positive (negative) type with respect to $N$ if
$$
\Delta\cap\sap(N)\subset\sp(N)\quad\big(\Delta\cap\sap(N)\subset\sm(N),\;\text{respectively}\big).
$$
A point $\la\in\sap(N)$ is called a {\it spectral point of definite type} of $N$ if it is either a spectral point of positive type or of negative type of $N$. Analogously, a set $\Delta\subset\C$ is said to be of definite type with respect to $N$ if it is either of positive or of negative type with respect to $N$.
\end{defn}

\begin{rem}
The notation of the spectrum of positive type is not consistent in the literature. In some works (see e.g.\ \cite{lamm,l,lmm}) it is denoted as above by $\sp(N)$; in others the authors write $\spp(N)$ instead (see e.g.\ \cite{abjt,pst}). Here we agree to follow \cite{lamm,l,lmm} and use the notation $\sp(N)$. This is also justified by the fact that $\spp(N)$ will be used for the spectral points of two-sided positive type defined below.
\end{rem}

It is immediately seen that, after a slight modification, Definition \ref{d:sp} can be formulated also for unbounded linear operators or relations. In fact, the spectral points of definite type were introduced and studied in \cite{abjt} for closed linear relations in Krein spaces. The following lemma is well-known (see \cite[Lemma 3.1]{abjt}).

\begin{lem}\label{l:sp_open}
The sets $\sp(N)$ and $\sm(N)$ are open in $\sap(N)$.
\end{lem}

\begin{lem}\label{l:proj}
Let $Q$ be a bounded projection in $\calH$ such that
$$
B\in L(\calH),\;NB=BN\quad\Lra\quad QB=BQ.
$$
Then $Q$ is normal. If, in addition, one of the following conditions
\begin{enumerate}
\item[{\rm (a)}] $\sap(N|Q\calH)\,\subset\,\sp(N)\cup\sm(N)$.
\item[{\rm (b)}] $Q\calH$ is uniformly positive or uniformly negative.
\end{enumerate}
holds, then $Q$ is selfadjoint.
\end{lem}
\begin{proof}
We have (for the second implication apply the adjoint)
$$
NN^+ = N^+N\,\Lra\,QN^+ = N^+Q\,\Lra\,NQ^+ = Q^+N\,\Lra\,QQ^+ = Q^+Q.
$$
Therefore, $Q$ as well as $P := Q - QQ^+$ are normal projections. Moreover, $P$ commutes with $N$, and we have $P^+P = 0$ so that the subspace $P\calH\subset Q\calH$ is neutral. Hence, if (b) holds, then $P=0$ follows immediately. If (a) is satisfied, then we have $\sap(N|P\calH)=\emptyset$ and thus also $P=0$.
\end{proof}

The next theorem was shown in \cite{lmm} in a somewhat more general situation. 

\begin{thm}\label{t:lmm}
Let $A$ be a bounded selfadjoint operator in the Krein space $(\calH,\product)$. If the interval $\Delta$ is of positive {\rm (}negative{\rm )} type with respect to $A$ then $A$ has a local spectral function $E$ of positive type {\rm (}negative type, respectively{\rm )} on $\Delta$. If $\delta\in\mathfrak B(\Delta)$ is compact then $E(\delta)\calH$ is the maximal spectral subspace of $A$ corresponding to $\delta$.
\end{thm}

Hereby, the {\it maximal spectral subspace} of a bounded operator $T$ in a Banach space $X$ {\it corresponding to the compact set $\Delta\subset\C$} is a closed $T$-invariant subspace $\calL_\Delta\subset X$ such that $\sigma(T|\calL_\Delta)\subset\Delta$ and $\calL\subset\calL_\Delta$ for any closed $T$-invariant subspace $\calL$ with $\sigma(T|\calL)\subset\Delta$. If such a subspace $\calL_\Delta$ exists, it is obviously unique.

In what follows we will deal with the question whether also a {\it normal} operator has a local spectral function of positive type on sets which are of positive type. In the next three lemmas we collect some necessary conditions. The first one is a direct consequence of Lemma \ref{l:proj}. The proof of the second lemma is straightforward and is left to the reader.

\begin{lem}\label{l:proj_sa}
If $N$ has a local spectral function $E$ of positive or negative type on the Borel set $\calS$, then for each $\Delta\in\mathfrak B(\calS)$ the projection $E(\Delta)$ is selfadjoint and commutes with both $N$ and $N^+$.
\end{lem}

For a set $\Delta\subset\C$ we define $\Delta^* := \{\ol\la : \la\in\Delta\}$.

\begin{lem}\label{l:lsf_N^+}
If $E$ is a local spectral function of positive {\rm (}negative{\rm )} type for $N$ on the Borel set $\calS$, then $E_+$, defined by
$$
E_+(\Delta) := E(\Delta^*),\quad\Delta\in\frakB(\calS^*),
$$
is a local spectral function of positive type {\rm (}negative type, respectively{\rm )} for $N^+$ on $\calS^*$.
\end{lem}

By $B_r(\la)$ we denote the disk with center $\la\in\C$ and radius $r > 0$.

\begin{lem}\label{l:lsf_nec}
If $N$ has a local spectral function of positive type on the open set $\calS$ then the following statements hold:
\begin{enumerate}
\item[{\rm (a)}] $\calS$ is of positive type with respect to $N$.
\item[{\rm (b)}] $\calS^*$ is of positive type with respect to $N^+$.
\item[{\rm (c)}] $\sap(N)\cap\calS = \sigma(N)\cap\calS$.
\item[{\rm (d)}] $\sap(N^+)\cap\calS^* = \sigma(N^+)\cap\calS^*$.
\item[{\rm (e)}] The approximate eigensequences for $N - \la$ and $N^+ - \ol\la$ coincide for each $\la\in\sap(N)\cap\calS$.
\end{enumerate}
\end{lem}
\begin{proof}
In view of Lemma \ref{l:lsf_N^+} it suffices to show only (a), (c) and that approximate eigensequences of $N - \la$ are also approximate eigensequences of $N^+ - \ol\la$ for $\la\in\sap(N)\cap\calS$.

Let $\la\in\calS\cap\sigma(N)$ and choose $\veps > 0$ such that $B_0 := \ol{B_\veps(\la)}$ is contained in $\calS$. We set $\calL_0 := E(B_0)$, where $E$ is the local spectral function of positive type of $N$ on $\calS$. As $E(B_0)$ is selfadjoint by Lemma \ref{l:proj_sa}, we have $\calL_1 := \calL_0^\gperp = (I - E(B_0))\calH$ and thus $\calH = \calL_0[\ds]\calL_1$. The subspace $\calL_0$ is $N$- and $N^+$-invariant. Hence, the same holds for $\calL_1$. Set $N_j := N|\calL_j$, $j=0,1$.

It follows from (S5) that $\la\in\rho(N_1)$. And as $\sigma(N) = \sigma(N_0)\cup\sigma(N_1)$, we conclude $\la\in\sigma(N_0)$. But $N_0$ is a normal operator in a Hilbert space by (S6) and thus $\la\in\sap(N_0)\,\subset\,\sap(N)$. This shows (c). Let $(x_n)\subset\calH$ be an approximate eigensequence for $N - \la$ and let $(x_{j,n})\subset\calL_j$, $j=0,1$, such that $x_n = x_{0,n} + x_{1,n}$, $n\in\N$. As $\la\in\rho(N_1)$, we conclude from $(N_1 - \la)x_{1,n}\to 0$ that $x_{1,n}\to 0$ as $n\to\infty$. Hence, from the uniform positivity of $\calL_0$ we obtain
$$
\liminf_{n\to\infty}\,[x_n,x_n] = \liminf_{n\to\infty}\,[x_{0,n},x_{0,n}] > 0,
$$
and (a) is proved. Moreover,
\begin{align*}
\|(N^+ - \ol\la)x_n\|
&\le \|(N^+ - \ol\la)x_{0,n}\| + \|(N^+ - \ol\la)x_{1,n}\|\\
&\le \delta\,[(N^+ - \ol\la)x_{0,n},(N^+ - \ol\la)x_{0,n}] + \|N^+ - \ol\la\|\,\|x_{1,n}\|\\
&=   \delta\,[(N - \la)x_{0,n},(N - \la)x_{0,n}] + \|N^+ - \ol\la\|\,\|x_{1,n}\|
\end{align*}
with some $\delta > 0$. This tends to zero as $n\to\infty$.
\end{proof}

The next lemma shows that parts of the necessary conditions in Lemma \ref{l:lsf_nec} are always satisfied for an open set which is of positive type with respect to $N$. By $\calL_\la(T)$ we denote the root subspace of $T\in L(\calH)$ corresponding to $\la\in\C$.

\begin{lem}\label{l:basic}
Let $\la\in\sp(N)$. Then the following statements hold.
\begin{enumerate}
\item[{\rm (i)}]   The approximate eigensequences for $N - \la$ are also approximate eigen\-sequences for $N^+ - \ol{\la}$.
\item[{\rm (ii)}]  $\ol\la\in\sap(N^+)$.
\item[{\rm (iii)}] $\ker(N - \la)\subset\ker(N^+ - \ol\la)$.
\item[{\rm (iv)}]  $\calL_\la(N) = \ker(N - \la)$.
\end{enumerate}
\end{lem}
\begin{proof}
Clearly, (ii) and (iii) follow from (i). So, let us show (i). To this end let $(x_n)\subset\calH$ be an approximate eigensequence for $N - \la$. Then
\begin{equation}\label{e:limit}
(N - \la)(N^+ - \ol\la)x_n\to 0\quad\text{as $n\to\infty$}.
\end{equation}
Suppose that $\limsup_{n\to\infty}\,\|(N^+ - \ol\la)x_n\| > 0$. Then there exists a subsequence $(x_{n_k})$ of $(x_n)$ and $\delta > 0$ such that $\|(N^+ - \ol\la)x_{n_k}\|\to\delta$ as $k\to\infty$. But as $\la\in\sp(N)$, it follows from \eqref{e:limit} that
$$
\liminf_{k\to\infty}\,[(N - \la)(N^+ - \ol\la)x_{n_k},x_{n_k}] = \liminf_{k\to\infty}\,[(N^+ - \ol\la)x_{n_k},(N^+ - \ol\la)x_{n_k}]\,>\,0,
$$
which contradicts \eqref{e:limit}. Therefore, $(N^+ - \ol\la)x_n\to 0$ as $n\to\infty$.

It remains to prove (iv). Let $x,u\in\calH$ such that $(N - \la)x = 0$ and $(N - \la)u = x$. Then (iii) yields $(N^+ - \ol\la)x = 0$ and thus
$$
[x,x] = [(N - \la)u,x] = [u,(N^+ - \ol\la)x] = 0,
$$
which implies $x = 0$ as $\la\in\sp(N)$.
\end{proof}

It follows from Lemmas \ref{l:sp_open} and \ref{l:basic} that the necessary condition (e) in Lemma \ref{l:lsf_nec} for the existence of a local spectral function of positive type for $N$ in an open neighborhood $\calS$ of $\la\in\sp(N)$ is satisfied if only the approximate eigensequences for $N^+ - \ol\la$ are also approximate eigensequences for $N - \la$. Obviously, this is equivalent to the following implication:
\begin{equation}\label{e:imp}
\la\in\sp(N)\quad\Lra\quad\ol\la\in\sp(N^+).
\end{equation}
We return to this problem in section \ref{s:spectral_set} and show there that \eqref{e:imp} is true if $(\calH,\product)$ is a Pontryagin space.

Motivated by Lemma \ref{l:lsf_nec}, we define a new class of spectral points for normal operators.

\begin{defn}
A point $\la\in\C$ is called a {\em spectral point of two-sided positive {\rm (}negative{\rm )} type} of the normal operator $N$ if
\begin{align*}
\la\in\sp(N)\quad &\text{and}\quad\ol\la\in\sp(N^+)\\
\big(\la\in\sm(N)\quad &\text{and}\quad\ol\la\in\sm(N^+),\,\text{ respectively}\big).
\end{align*}
The set of all spectral points of two-sided positive (negative) type of $N$ is denoted by $\spp(N)$ {\rm (}$\smm(N)$, respectively{\rm )}. A set $\Delta\subset\C$ is said to be of two-sided positive {\rm (}negative{\rm )} type with respect to $N$ if
$$
\Delta\cap\sigma(N)\subset\spp(N) \quad\big(\Delta\cap\sigma(N)\subset\smm(N),\;\text{respectively}\big).
$$
\end{defn}

In the sequel we restrict ourselves to the investigation of the spectrum of two-sided positive type. Similar results hold for spectral points and sets of two-sided negative type.

\begin{rem}
In the case of a selfadjoint operator $N$ the sets $\spp(N)$ and $\sp(N)$ coincide and are contained in $\R$ (see, e.g., \cite{lmm}).
\end{rem}

Remark \ref{r:set_tspt} and Lemma \ref{l:basic2} below directly follow from Lemma \ref{l:basic}.

\begin{rem}\label{r:set_tspt}
For a set $\Delta$ we have $\Delta\cap\sigma(N)\subset\spp(N)$ if and only if
$$
\Delta\cap\sap(N)\subset\sp(N)\quad\text{and}\quad\Delta^*\cap\sap(N^+)\subset\sp(N^+).
$$
\end{rem}

\begin{lem}\label{l:basic2}
Let $\la\in\spp(N)$. Then the following holds.
\begin{enumerate}
\item[{\rm (i)}]   The approximate eigensequences for $N - \la$ and $N^+ - \ol\la$ coincide.
\item[{\rm (ii)}]  $\ker(N - \la) = \ker(N^+ - \ol\la) = \calL_\la(N) = \calL_{\ol\la}(N^+)$.
\end{enumerate}
\end{lem}

Note that for each $\la\in\C$ the operator
$$
A(\la) := (N^+ - \ol\la)(N - \la)
$$
is selfadjoint (in the Krein space $(\calH,\product)$). The following lemma shows that the spectrum of two-sided positive type of $N$ is closely related to the sign type behaviour of the zero point with respect to the operators $A(\la)$. This correspondence will serve as the starting point for the construction of the local spectral function in section \ref{s:lsf}.

\begin{lem}\label{l:selfadjoint}
For all $\la\in\C$ we have
$$
\la\in\spp(N)\quad\Llra\quad 0\in\sp(A(\la)).
$$
\end{lem}
\begin{proof}
Let $\la\in\spp(N)$. Then, clearly, $0\in\sap((N^+ - \ol\la)(N - \la))$. Let $(x_n)$ be an approximate eigensequence for $(N^+ - \ol\la)(N - \la)$. Suppose that there exists a subsequence $(x_{n_k})$ of $(x_n)$ such that $\lim_{k\to\infty}\,\|(N - \la)x_{n_k}\| > 0$. Then from $\ol\la\in\sp(N^+)$ we obtain a contradiction:
$$
0 = \liminf_{k\to\infty}\,[(N^+ - \ol\la)(N - \la)x_{n_k},x_{n_k}] = \liminf_{k\to\infty}\,[(N - \la)x_{n_k},(N - \la)x_{n_k}] > 0.
$$
Therefore, $(N - \la)x_n\to 0$ as $n\to\infty$, and thus $\liminf_{n\to\infty}\,[x_n,x_n] > 0$ as $\la\in\sp(N)$. Conversely, assume that $0\in\sp((N^+ - \ol\la)(N - \la))$. Then $\la\in\sap(N)$ or $\ol\la\in\sap(N^+)$. Assume, e.g., that $\la\in\sap(N)$ and let $(x_n)$ be an approximate eigensequence for $N - \la$. Then $(x_n)$ is also an approximate eigensequence for $(N^+ - \ol\la)(N - \la)$ and thus $\liminf_{n\to\infty}\,[x_n,x_n] > 0$ follows. Hence, $\la\in\sp(N)$ and therefore $\ol\la\in\sap(N^+)$ by Lemma \ref{l:basic}. A similar reasoning as above shows $\ol\la\in\sp(N^+)$.
\end{proof}

For a compact set $K\subset\C$ and $\veps > 0$ we set
$$
B_\veps(K) := \bigcup_{\la\in K}\,B_\veps(\la).
$$

\begin{lem}\label{l:A_compact}
Let $K\subset\C$ be a compact set which is of two-sided positive type with respect to $N$, i.e.
$$
K\cap\sigma(N)\,\subset\,\spp(N).
$$
Then there exist $\veps_0,\delta_0 > 0$ such that for all $\mu\in [0,\veps_0^2]$, all $\la\in\ol{B_{\veps_0}(K)}$ and all $x\in\calH$ the following implications hold:
\begin{enumerate}
\item[{\rm (a)}] $\|(A(\la) - \mu)x\|\le\veps_0\|x\|\quad\Lra\quad [x,x]\ge\delta_0\|x\|^2$.
\item[{\rm (b)}] $\|(N - \la)x\|\le\veps_0\|x\|\;\;\text{or}\;\;\|(N^+ - \ol\la)x\|\le\veps_0\|x\|\quad\Lra\quad [x,x]\ge\delta_0\|x\|^2$.
\end{enumerate}
In particular,
$$
\ol{B_{\veps_0}(K)}\cap\sigma(N)\,\subset\,\spp(N),
$$
and for all $\la\in\ol{B_{\veps_0}(K)}$ we have
$$
[0,\veps_0^2]\cap\sigma(A(\la))\,\subset\,\sp(A(\la)).
$$
\end{lem}
\begin{proof}
Assume that it is not true that there are $\veps_1,\delta_1 > 0$ such that for all $(\mu,\la,x)\in [0,\veps_1^2]\times\ol{B_{\veps_1}(K)}\times\calH$ we have
$$
\|(A(\la) - \mu)x\|\le\veps_1\|x\|\quad\Lra\quad [x,x]\ge\delta_1\|x\|^2.
$$
Then for each $n\in\N$ there exist $\mu_n\in [0,\frac{1}{n^2}]$, $\la_n\in\ol{B_{\frac 1 n}(K)}$ and $x_n\in\calH$ with $\|x_n\| = 1$ such that $\|(A(\la_n) - \mu_n)x_n\|\le\frac 1 n$ and $[x_n,x_n] < \frac 1 n$. As $\ol{B_1(K)}$ is compact and $\la_n\in \ol{B_1(K)}$ for all $n\in\N$, there exists a subsequence $(\la_{n_k})$ of $(\la_n)$ which converges to some $\la_0\in\ol{B_1(K)}$. But $\la_{n_k}\in\ol{B_{\frac 1{n_k}}(K)}$ so that $\la_0\in K$. It follows that
\begin{align*}
\|A(\la_0)x_{n_k}\|
&\le\|(A(\la_0) - A(\la_{n_k}))x_{n_k}\| + \|(A(\la_{n_k}) - \mu_{n_k})x_{n_k}\| + |\mu_{n_k}|\\
&\le\|A(\la_0) - A(\la_{n_k})\| + \frac 1 {n_k} + \frac 1 {n_k^2}\,.
\end{align*}
Since the function $A : \C\to L(\calH)$ is continuous, it follows that $A(\la_0)x_{n_k}\to 0$ as $k\to\infty$. This implies $\la_0\in\sigma(N)$ and hence $\la_0\in\spp(N)$. By Lemma \ref{l:selfadjoint}, $0\in\sp(A(\la_0))$ which is a contradiction to $[x_n,x_n] < \frac 1 n$.

In a similar way it can be shown that there exist $\veps_2,\delta_2 > 0$ such that for all $(\la,x)\in \ol{B_{\veps_2}(K)}\times\calH$ we have
$$
\|(N - \la)x\|\le\veps_2\|x\|\;\;\text{or}\;\;\|(N^+ - \ol\la)x\|\le\veps_2\|x\|\quad\Lra\quad [x,x]\ge\delta_2\|x\|^2.
$$
With $\veps_0 := \min\{\veps_1,\veps_2\}$ and $\delta_0 := \min\{\delta_1,\delta_2\}$ the assertion follows.
\end{proof}

\begin{cor}\label{c:spp_open}
The set $\spp(N)$ is open in $\sigma(N)$.
\end{cor}

\section{The local spectral function}\label{s:lsf}
The main result of this section is the following theorem.

\begin{thm}\label{t:lsf_pt}
Let $N$ be a bounded normal operator in the Krein space $(\calH,\product)$ and let $\calS\subset\C$ be a Borel set which is of two-sided positive type with respect to $N$. Then $N$ has a local spectral function $E$ of positive type on $\calS$. If $\Delta\in\mathfrak B(\calS)$ is compact, then $E(\Delta)\calH$ is the maximal spectral subspace of $N$ corresponding to $\Delta$.
\end{thm}

As $\spp(N)$ is open in $\sigma(N)$, it is sufficient to prove Theorem \ref{t:lsf_pt} only for open sets $\calS$. For the proof we need two preparatory lemmas.

\begin{lem}\label{l:I}
Let $K\subset\C$ be a compact set which is of two-sided positive type with respect to $N$. Then there exists $\veps_0 > 0$ such that for each disk $B_\veps(\la)\subset B_{\veps_0}(K)$ with radius $\veps\in (0,\veps_0]$ there exists a closed subspace $\calL_{\la,\veps}\subset\calH$ with the following properties:
\begin{enumerate}
\item[{\rm (a)}] $\left(\calL_{\la,\veps},\product\right)$ is a Hilbert space which is both $N$- and $N^+$-invariant.
\item[{\rm (b)}] $\sigma(N|\calL_{\la,\veps})\,\subset\,\sigma(N)\cap\ol{B_\veps(\la)}$.
\item[{\rm (c)}] If $\calM\subset\calH$ is a closed subspace which is both $N$- and $N^+$-invariant such that $(\calM,\product)$ is a Hilbert space with
$$
\sigma(N|\calM)\subset\ol{B_\veps(\la)},
$$
then $\calM\subset\calL_{\la,\veps}$.
\item[{\rm (d)}] If $B\in L(\calH)$ commutes with both $N$ and $N^+$, then $\calL_{\la,\veps}$ and $\calL_{\la,\veps}^\gperp$ both are $B$-invariant.
\item[{\rm (e)}] $\sigma(N)\cap B_\veps(\la)\neq\emptyset\quad\Lra\quad\calL_{\la,\veps}\neq\{0\}$.
\end{enumerate}
\end{lem}
\begin{proof}
Choose $\veps_0 > 0$ according to Lemma \ref{l:A_compact} and let $\la\in\C$ and $\veps\in (0,\veps_0]$ such that $B_\veps(\la)\subset B_{\veps_0}(K)$. By Lemma \ref{l:A_compact} we have
\begin{equation}\label{e:vepses2}
[0,\veps^2]\cap\sigma(A(\la))\subset\sp(A(\la))\quad\text{and}\quad\ol{B_\veps(\la)}\cap\sigma(N)\subset\spp(N).
\end{equation}
By Lemma \ref{l:sp_open} there exists $\delta > 0$ such that $[-\delta,0)\cap\sigma(A(\la))\subset\sp(A(\la))$. Due to Theorem \ref{t:lmm} the operator $A := A(\la)$ has a local spectral function $E_A$ of positive type on $[-\delta,\veps^2]$. Due to (S3) and (S6) the subspace
$$
\calL_{\la,\veps} := E_A([-\delta,\veps^2])\calH
$$
is uniformly positive as well as $N$- and $N^+$-invariant. Therefore, the restriction $N|\calL_{\la,\veps}$ is a normal operator in the Hilbert space $(\calL_{\la,\veps},\product)$ with the adjoint $N^+|\calL_{\la,\veps}$ and
$$
A|\calL_{\la,\veps} = \big((N|\calL_{\la,\veps})^+ - \la\big)\big((N|\calL_{\la,\veps}) - \la\big).
$$
Hence, $A|\calL_{\la,\veps}$ is a non-negative selfadjoint operator in a Hilbert space which implies
$$
(-\delta,0)\subset\rho(A).
$$
Let $f(z) := (\ol z - \ol{\la})(z - \la) = |z - \la|^2$, $z\in\C$. This is a continuous function on $\C$, and we obtain
$$
\sigma(A|\calL_{\la,\veps}) = \sigma(f(N|\calL_{\la,\veps})) = f(\sigma(N|\calL_{\la,\veps})).
$$
Therefore, $z\in\sigma(N|\calL_{\la,\veps})$ implies $f(z)\in [0,\veps^2]$ and thus $z\in\ol{B_\veps(\la)}$. Since $\sigma(N|\calL_{\la,\veps}) = \sap(N|\calL_{\la,\veps})\subset\sigma(N)$, (b) is proved.

A subspace $\calM$ as in (c) is obviously $A$-invariant, and we have
$$
\sigma(A|\calM) = \sigma(f(N|\calM)) = f(\sigma(N|\calM))\,\subset\,f\left(\ol{B_\veps(\la)}\right)\,\subset\,[0,\veps^2].
$$
And since $\calL_{\la,\veps}$ is the maximal spectral subspace of $A$ corresponding to $[0,\veps^2]$, it follows that $\calM\subset\calL_{\la,\veps}$.

If $B\in L(\calH)$ as in (d), then $BA = AB$ and (d) follows from (S3).

For the proof of (e) assume that $\calL_{\la,\veps} = \{0\}$. Then $E_A([-\delta,\veps^2]) = 0$ and (S5) implies $(-\delta,\veps^2)\subset\rho(A)$. If $z\in\sigma(N)\cap B_\veps(\la)$, then $z\in\spp(N)$ and there exists an approximate eigensequence for both $N - z$ and $N^+ - \ol z$. Consequently, $(A - |\la - z|^2)x_n\to 0$ as $n\to\infty$ which contradicts $(-\delta,\veps^2)\subset\rho(A)$. Therefore, $\sigma(N)\cap B_\veps(\la) = \emptyset$.
\end{proof}

Note that the subspaces $\calL_{\la,\veps}$ in Lemma \ref{l:I} are uniquely determined by (a)--(c).

\begin{lem}\label{l:II}
Let $K$ and $\veps_0$ be as in Lemma {\rm\ref{l:I}}. Let $\Delta_1,\ldots,\Delta_m\subset\ol{B_{\veps_0}(K)}$ be closed sets such that for each $j\in\{1,\ldots,m\}$ there exists a closed subspace $\calL_j\subset\calH$ with
\begin{enumerate}
\item[$({\rm a}_j)$] $\left(\calL_j,\product\right)$ is a Hilbert space which is both $N$- and $N^+$-invariant.
\item[$({\rm b}_j)$] $\sigma(N|\calL_j)\,\subset\,\sigma(N)\cap\Delta_j$.
\item[$({\rm c}_j)$] If $\calM\subset\calH$ is a subspace which is both $N$- and $N^+$-invariant such that $(\calM,\product)$ is a Hilbert space with
$$
\sigma(N|\calM)\subset\Delta_j,
$$
then $\calM\subset\calL_j$.
\item[$({\rm d}_j)$] If $B\in L(\calH)$ commutes with both $N$ and $N^+$, then $\calL_j$ and $\calL_j^\gperp$ both are $B$-invariant.
\end{enumerate}
Then the subspace $\calL_0 := \calL_1 + \ldots + \calL_m$ is closed and satisfies $({\rm a}_0)$--$({\rm d}_0)$, where $\Delta_0 := \Delta_1\cup\ldots\cup\Delta_m$.
\end{lem}
\begin{proof}
We only prove Lemma \ref{l:II} for $m=2$. The general case then follows by induction. For $j\in\{1,2\}$ denote by $E_j$ the $\product$-orthogonal projection onto $\calL_j$ and define
$$
E_0 := E_1 + E_2 - E_1E_2 = E_1 + (I - E_1)E_2.
$$
From $({\rm a}_1)$ it follows that $E_1$ commutes with $N$ and $N^+$. By $({\rm d}_2)$, $E_1E_2 = E_2E_1$. Hence, $E_0$ is a selfadjoint projection, and the following relation holds:
$$
\calL_0 = \calL_1 + \calL_2 = \calL_1\,[\ds]\,\left(\calL_1^\gperp\cap\calL_2\right) = E_0\calH.
$$
By \cite[Lemma I.5.3]{l}, $(\calL_0,\product)$ is a Hilbert space. Thus, $({\rm a}_0)$ holds, and $({\rm b}_0)$ as well as $({\rm d}_0)$ are easily verified.

Let $\calM$ be a subspace as in $({\rm c}_0)$. Then $N|\calM$ is a normal operator in the Hilbert space $(\calM,\product)$. Let $F$ be its spectral measure. Then
$$
\calM = F(\Delta_1)\calM\,[\ds]\,(I_\calM - F(\Delta_1))\calM.
$$
We have $\sigma(N|F(\Delta_1)\calM)\subset\Delta_1$ and
$$
\sigma(N|(I_\calM - F(\Delta_1))\calM)\subset\ol{\sigma(N|\calM)\setminus\Delta_1}\subset\ol{(\Delta_1\cup\Delta_2)\setminus\Delta_1}\subset\Delta_2.
$$
Hence, from (c$_1$) and (c$_2$) we conclude $F(\Delta_1)\calM\subset\calL_1$, $(I_\calM - F(\Delta_1))\calM\subset\calL_2$ and therefore $\calM\subset\calL_0$.
\end{proof}

A proof of the following lemma can be found in \cite{rr} (see also \cite{pst}).

\begin{lem}[Rosenblum's Corollary]\label{l:rosenblum}
Let $\calX$ and $\calY$ be Banach spaces and let $S\in L(\calX)$ and $T\in L(\calY)$. If $\sigma(S)\cap\sigma(T) = \emptyset$, then for every $Z\in L(\mathcal Y,\mathcal X)$ the operator equation
$$
SX - XT = Z
$$
has a unique solution $X\in L(\mathcal Y,\mathcal X)$. In particular, $SX = XT$ implies $X = 0$.
\end{lem}

We are now prepared to prove Theorem \ref{t:lsf_pt}. By $\Delta^i$ we denote the interior of a subset $\Delta\subset\C$.

\begin{proof}[Proof of Theorem {\rm\ref{t:lsf_pt}}]
The proof is divided into three steps. In the first two steps it is shown that Theorem {\rm\ref{t:lsf_pt}} holds for compact sets $\calS = K$. More precisely, in step 1 it is shown that there exists a spectral subspace $\calL_0$ for $N$ corresponding to a compact set $\Delta_0$ containing $K$ which has the properties (a$_0$)--(d$_0$) in Lemma \ref{l:II} and (i)--(iii) below. In the second step the local spectral function of $N$ on $K$ is defined via the orthogonal projection onto $\calL_0$ and the spectral measure of the normal operator $N|\calL_0$ in the Hilbert space $(\calL_0,\product)$. In the last step we prove that Theorem \ref{t:lsf_pt} holds for open sets $\calS$.

{\bf 1.} Let $K$ be a compact set of two-sided positive type with respect to $N$ and let $\veps_0 > 0$ be as in Lemma \ref{l:I}. Then choose some $\veps_1\in (0,\veps_0)$ and $\la_1,\dots,\la_m\in K$ such that
$$
K\,\subset\,\bigcup_{j=1}^m\,B_{\veps_1}(\la_j)\,\subset\,\bigcup_{j=1}^m\,\ol{B_{\veps_1}(\la_j)}\,=:\,\Delta_0 \,\subset\,B_{\veps_0}(K).
$$
By Lemma \ref{l:I} and Lemma \ref{l:II} there exists a closed subspace $\calL_0\subset\calH$ satisfying $({\rm a}_0)$--$({\rm d}_0)$ in Lemma \ref{l:II}. We will show that $\calL_0$ also has the following properties:
\begin{enumerate}
\item[(i)]   $\sigma(N|\calL_0^\gperp)\,\subset\,\ol{\sigma(N)\setminus\Delta_0}$.
\item[(ii)]  If $B\in L(\calH)$ with $BN=NB$, then $\calL_0$ and $\calL_0^\gperp$ are $B$-invariant.
\item[(iii)] $\calL_0$ is the maximal spectral subspace of $N$ corresponding to $\Delta_0$.
\end{enumerate}
First of all we prove
\begin{equation}\label{e:big_pp}
\ol{B_{\veps_0}(K)}\cap\sigma(N|\calL_0^\gperp)\,\subset\,\spp(N|\calL_0^\gperp).
\end{equation}
Since $\ol{B_{\veps_0}(K)}\cap\sigma(N)\subset\spp(N)$, it suffices to show that
$$
\ol{B_{\veps_0}(K)}\cap\sigma(N|\calL_0^\gperp)\,\subset\,\sap(N|\calL_0^\gperp),
$$
cf.\ Lemma \ref{l:basic2}(i). Let $\la\in\ol{B_{\veps_0}(K)}\,\cap\,\sigma(N|\calL_0^\gperp)$ and suppose that $\la$ is not an element of $\sap(N|\calL_0^\gperp)$. Then $\ol\la\in\sigma_p(N^+|\calL_0^\gperp)$, from which $\la\in\sigma(N)$ follows. But this implies $\la\in\spp(N)$ and therefore $\la\in\sigma_p(N|\calL_0^\gperp)$ (see Lemma \ref{l:basic2}(ii)). A contradiction.

Let $\la\in\C\setminus\ol{\sigma(N)\setminus\Delta_0}$. Then $\sigma(N)\setminus\Delta_0$ does not accumulate to $\la$ which means that there exists $\veps' > 0$ such that $\sigma(N)\cap B_{\veps'}(\la)\subset\Delta_0$. Due to \eqref{e:big_pp} and Lemma \ref{l:I} there exist $\veps\in (0,\veps')$ and a closed $N$- and $N^+$-invariant subspace $\calM\subset\calL_0^\gperp$ such that $(\calM,\product)$ is a Hilbert space and $\sigma(N|\calM)\subset\sigma(N|\calL_0^\gperp)\cap\ol{B_\veps(\la)}$. As $\sigma(N|\calL_0^\gperp)\subset\sigma(N)$, we have $\sigma(N|\calM)\subset\Delta_0$. From (c$_0$) we conclude $\calM\subset\calL_0$. But $\calM\subset\calL_0^\gperp$ and thus $\calM = \{0\}$. From Lemma \ref{l:I}(e) we obtain $B_\veps(\la)\subset\rho(N|\calL_0^\gperp)$ which proves (i).

For the proofs of (ii) and (iii) let $(\delta_n)$ be a sequence of positive numbers such that
$$
\veps_0 > \delta_1 > \delta_2 > \ldots > \veps_1\quad\text{and}\quad\delta_n\downto\veps_1\;\text{ as }n\to\infty.
$$
Set
$$
\Delta_n := \bigcup_{j=1}^m\,\ol{B_{\delta_n}(\la_j)}\,\subset\, B_{\veps_0}(K).
$$
Then (recall that $\Delta_0$ was defined similarly)
$$
\Delta_0 = \bigcap_{n=1}^\infty\,\Delta_n\quad\text{and}\quad\Delta_0\subset\Delta_n^i\;\text{ for every }n\in\N.
$$
For $n\in\N$ by $\calL_n$ we denote the closed subspace which satisfies $({\rm a}_n)$--$({\rm d}_n)$ in Lemma \ref{l:II}. We have
\begin{equation}\label{e:alltogether}
\calL_{n+1}\,\subset\,\calL_n,\;\;\calL_0\,\subset\,\calL_n\;\;\text{for all $n\in\N\setminus\{0\}$ }\quad\text{and}\quad\calL_0 = \bigcap_{n=1}^\infty\,\calL_n.
\end{equation}
Indeed, the first two inclusions follow immediately from the properties $({\rm c}_k)$, $k\in\N$, in Lemma \ref{l:II}. Hence, $\calL_0\subset\bigcap_{n=1}^\infty\,\calL_n =: \calM$, and it is not difficult to see that $\sigma(N|\calM)\subset\Delta_0$ (consider $\la\notin\Delta_0$ and show $\la\in\rho(N|\calM)$). The last relation in \eqref{e:alltogether} follows now from $({\rm c}_0)$.

Let $B\in L(\calH)$ such that $BN=NB$ and set $B_0 := B|\calL_0\in L(\calL_0,\calH)$. By $E_n$ we denote the $\product$-orthogonal projection in $\calH$ onto $\calL_n$, $n\in\N$. As $\calL_n$ and $\calL_n^\gperp$ both are $N$-invariant, $E_n$ commutes with $N$. Hence, the following relation holds:
$$
\left(N|\calL_n^\gperp\right)\big(\left(I - E_n\right)B_0\big) = (I - E_n)NB_0 = \big(\left(I - E_n\right)B_0\big)\big(N|\calL_0\big).
$$
By (i) and $({\rm b}_0)$ the spectra of $N|\calL_n^\gperp$ and $N|\calL_0$ are disjoint. Thus, due to Rosenblum's Corollary (Lemma \ref{l:rosenblum}) it follows that $(I - E_n)B_0 = 0$, or equivalently, $B\calL_0\subset\calL_n$ for every $n\in\N$. By virtue of \eqref{e:alltogether}, $\calL_0$ is $B$-invariant. Similarly, one shows $B\calL_n^\gperp\subset\calL_0^\gperp$ for each $n\in\N$. It is easy to see that
$$
\cls\left\{\calL_n^\gperp : n\in\N\right\}^\gperp = \bigcap_{n=1}^\infty\,\calL_n = \calL_0.
$$
Hence, $B\calL_0^\gperp\subset\calL_0^\gperp$ follows immediately from
$$
\cls\left\{\calL_n^\gperp : n\in\N\right\} = \calL_0^\gperp.
$$
Let $\calM\subset\calH$ be a closed $N$-invariant subspace such that $\sigma(N|\calM)\subset\Delta_0$. Then from Rosenblum's Corollary and the relation
$$
\left(N|\calL_n^\gperp\right)\,\left(\big(I - E_n\big)\big|\calM\right) = \left(\big(I - E_n\big)\big|\calM\right)\,\big(N|\calM\big)
$$
we conclude $\calM\subset\calL_n$ for all $n\in\N$, and thus $\calM\subset\calL_0$, which shows (iii).

{\bf 2.} Let us complete the proof of Theorem \ref{t:lsf_pt} for $\calS = K$. Let $\Delta_0$ and $\calL_0$ be as in step 1. The operator $N_0 := N|\calL_0$ is a normal operator in the Hilbert space $(\calL_0,\product)$ and has therefore a spectral measure $E_0$. By $Q$ we denote the $\product$-orthogonal projection onto $\calL_0$ and define
$$
E(\Delta) := E_0(\Delta)Q,\quad\Delta\in\frakB(K).
$$
For each $\Delta\in\frakB(K)$ the operator $E(\Delta)$ is a selfadjoint projection, and it is easily seen that $E$ has the properties (S1)--(S3) and (S6) in Definition \ref{d:lsf_pt}. Let $\Delta\in\frakB(K)$. Then
\begin{align*}
\sigma(N|E(\Delta)\calH)
&= \sigma(N_0|E_0(\Delta)\calL_0)\subset\ol{\sigma(N_0)\cap\Delta}\\
&\subset\ol{\sigma(N)\cap\Delta_0\cap\Delta} = \ol{\sigma(N)\cap\Delta}.
\end{align*}
And since
$$
(I - E(\Delta))\calH = \calL_0^\gperp\,[\ds]\,\left(\big(E_0(\Delta)\calL_0\big)^\gperp\cap\calL_0\right),
$$
we have
\begin{align*}
\sigma(N|(I - E(\Delta))\calH)
&= \sigma(N|\calL_0^\gperp)\,\cup\,\sigma\left(N_0\big|\big(E_0(\Delta)\calL_0\big)^\gperp\cap\calL_0\right)\\
&\subset\,\ol{\sigma(N)\setminus\Delta_0}\,\cup\,\ol{\sigma(N_0)\setminus\Delta}\\
&\subset\,\ol{\sigma(N)\setminus\Delta}.
\end{align*}
Moreover, if $\Delta\subset K$ is closed, then $E(\Delta)\calH$ is the maximal spectral subspace of $N$ corresponding to $\Delta$ (we say that $E$ has the property (M)): if $\calM\subset\calH$ is a closed $N$-invariant subspace such that $\sigma(N|\calM)\subset\Delta$, then $\calM\subset\calL_0$ by (iii) in step 1 and hence, we have $\sigma(N_0|\calM)\subset\Delta$. From this and the properties of the spectral measure $E_0$ of $N_0$ we obtain $\calM\subset E_0(\Delta)\calL_0 = E_0(\Delta)Q\calH = E(\Delta)\calH$. In particular, this shows that the definition of $E$ does not depend on the choice of $\veps_0$ and $\veps_1$. Indeed, if $\wt E$ is another local spectral function of positive type for $N$ on $K$ with the property (M), then $\wt E(\Delta) = E(\Delta)$ for all closed sets $\Delta\subset K$. And as the system of the closed subsets of $K$ is a generator of the $\sigma$-algebra $\mathfrak B(K)$ which is stable with respect to intersections, $\wt E = E$ follows.

{\bf 3.} Finally, we prove that Theorem \ref{t:lsf_pt} holds for open sets $\calS$. Clearly, it is no restriction to assume that $\calS$ is bounded. For a closed set $K\subset\calS$ denote by $E_K$ the local spectral function of positive type of $N$ on $K$ defined in the previous steps. Let $K_1$ and $K_2$ be two closed sets in $\calS$ such that $K_1\subset K_2$. Then $E_{K_2}|\mathfrak B(K_1)$ is a local spectral function of positive type for $N$ on $K_1$ with the property (M). Hence, $E_{K_2}|\mathfrak B(K_1) = E_{K_1}$. Therefore,
$$
E(\Delta) := E_{\ol\Delta}(\Delta),\quad\Delta\in\mathfrak B(\calS)
$$
defines a local spectral function of positive type for $N$ on $\calS$ with the property (M). The theorem is proved.
\end{proof}

The following corollary is a direct consequence of Theorem \ref{t:lsf_pt}.

\begin{cor}\label{c:res}
Let $\la_0\in\spp(N)$ be an accumulation point of $\rho(N)$. Then there exist $\veps > 0$ and $C > 0$ such that for all $\la\in B_\veps(\la_0)\cap\rho(N)$ we have
$$
\|(N - \la)^{-1}\|\,\le\,\frac{C}{\operatorname{dist}(\la,\sigma(N))}.
$$
In particular, an isolated spectral point of $N$ which is of two-sided positive type is a pole of order one of the resolvent of $N$.
\end{cor}
\begin{proof}
Choose $\veps > 0$ such that $B_0 := \ol{B_{2\veps}(\la_0)}$ is of two-sided positive type with respect to $N$. Denote by $E$ the local spectral function of $N$ on $B_0$ and set $\calL_0 := E(B_0)\calH$. Then $B_0^i\subset\rho(N|\calL_0^\gperp)$. Hence, there exists $C_1 > 0$ such that
$$
\left\|\left((N|\calL_0^\gperp) - \la\right)^{-1}\right\|\,\le\,C_1\quad\text{ for all }\la\in B_\veps(\la_0).
$$
The restriction of $N$ to $\calL_0$ is a normal operator in a Hilbert space. Therefore, for any $x\in\calL_0$ and $\la\in B_\veps(\la_0)\cap\rho(N)$ we have the well-known inequality
$$
[(N - \la)x,(N - \la)x]\,\ge\,\operatorname{dist}(\la,\sigma(N|\calL_0))^2\,[x,x].
$$
As the subspace $\calL_0$ is uniformly positive, this implies
$$
\|(N - \la)x\|^2\,\ge\,\operatorname{dist}(\la,\sigma(N))^2\cdot\delta\|x\|^2,
$$
with some $\delta > 0$, and the assertion follows.
\end{proof}

\section{Spectral sets which are of positive type}\label{s:spectral_set}
Let $\sigma$ be a spectral set for $N$ (i.e.\ a subset of $\sigma(N)$ which is both open and closed in $\sigma(N)$) which is of positive type with respect to $N$. In \cite{pst} it was shown that the Riesz-Dunford spectral subspace of $N$ corresponding to $\sigma$ is uniformly positive if the spectrum of the imaginary part $\Im N = \frac 1 {2i}(N - N^+)$ is real and if there exist $C > 0$ and $m\in\N$ such that
\begin{equation}\label{e:growth}
\|(\Im N - \la)^{-1}\|\,\le\,\frac C {|\Im\la|^m}
\end{equation}
holds for all non-real $\la$ in a neighborhood of $\sigma(\Im N)$. The same holds if the above conditions are satisfied for the real part $\Re N = \frac 1 2 (N + N^+)$ instead for $\Im N$. The following theorem shows that these assumptions are redundant.

\begin{thm}\label{t:spectral_set}
Let $\sigma$ be a spectral set for $N$, let $Q$ be the Riesz-Dunford projection of $N$ corresponding to $\sigma$ and assume that
\begin{equation}\label{e:sp_inc}
\sigma\cap\sap(N)\,\subset\,\sp(N).
\end{equation}
Then $Q$ is selfadjoint and $Q\calH$ is uniformly positive. In particular, $N|Q\calH$ is a normal operator in the Hilbert space $(Q\calH,\product)$.
\end{thm}
\begin{proof}
The projection $Q$ is selfadjoint by Lemma \ref{l:proj} (see also \cite[Lemma 5.1]{pst}). This implies that the inner product space $(Q\calH,\product)$ is a Krein space which is invariant with respect to both $N$ and $N^+$. Moreover, we have
$$
(N|Q\calH)^+ = N^+|Q\calH
$$
and $\sap(N|Q\calH) = \sp(N|Q\calH)$. It is therefore no restriction to assume $\calH = Q\calH$ and $\sap(N) = \sp(N)$. In view of Remark \ref{r:set_tspt} and Theorem \ref{t:lsf_pt} it only remains to show that $\C$ is of positive type with respect to $N^+$, i.e.\ $\sap(N^+)\subset\sp(N^+)$. Let $\ol\la\in\sap(N^+)$. We have to show that the approximate eigensequences for $N^+ - \ol\la$ are also approximate eigensequences for $N - \la$. To this end we introduce the Banach space
$$
\wt\calH := \ell^\infty(\calH)/c_0(\calH),
$$
where by $\ell^\infty(\calH)$ we denote the space of all bounded sequences $(x_n)$ in $\calH$ with norm $\|(x_n)\|_{\ell^\infty(\calH)} = \sup_n\,\|x_n\|$, and $c_0(\calH)$ is the closed subspace of $\ell^\infty(\calH)$ consisting of the sequences $(x_n)$ with $\lim_n\,\|x_n\| = 0$. It is not difficult to show that the norm of a coset $[(x_n)]\in\wt\calH$ is given by
$$
\|[(x_n)]\|_{\wt\calH} = \limsup_{n\to\infty}\,\|x_n\|.
$$
Consider the operators $\wt N$ and $\wt{N^+}$ in $\wt\calH$, defined by
$$
\wt N[(x_n)] := [(Nx_n)]\quad\text{and}\quad\wt{N^+}[(x_n)] := [(N^+x_n)], \;\;[(x_n)]\in\wt\calH.
$$
The operators $\wt N$ and $\wt{N^+}$ are well-defined and $\wt N,\wt{N^+}\in L(\wt\calH)$ holds where $\|\wt N\|\le\|N\|$ and $\|\wt{N^+}\|\le\|N^+\|$. As $N$ and $N^+$ commute, also $\wt N$ and $\wt{N^+}$ commute.

Observe that if $(x_n)$ is an approximate eigensequence for $N - \la$ then $[(x_n)]\in\ker(\wt N - \la)$. Conversely, if $(x_n)\subset\calH$ with $\|x_n\| = 1$ for $n\in\N$ such that $[(x_n)]\in\ker(\wt N - \la)$, then $(x_n)$ is an approximate eigensequence for $N - \la$. An analogue correspondence holds for $N^+ - \ol\la$ and $\wt{N^+} - \ol\la$. Therefore, we have to show that
$$
\ker(\wt{N^+} - \ol\la)\,\subset\,\ker(\wt N - \la).
$$
To see this, we define the subspace
$$
\calM := \ol{(\wt N - \la)\ker(\wt{N^+} - \ol\la)}.
$$
This subspace is $\wt N$-invariant. We are done if we can show that $\calM = \{0\}$, or equivalently, $\sap(\wt N|\calM) = \emptyset$. Thus, suppose that there exist a sequence $(\wt x_m)\subset\calM$ and $\mu\in\C$ such that
$$
\|\wt x_m\|_{\wt\calH} = 1\quad\text{and}\quad\lim_{m\to\infty}\,\|(\wt N - \mu)\wt x_m\|_{\wt\calH} = 0.
$$
For each $m\in\N$ there exists a sequence $(x_n^{(m)})\in\ell^\infty(\calH)$ such that $\wt x_m = [(x_n^{(m)})]$. Let $m\in\N$. As $[(x_n^{(m)})]\in\calM$, there exists $[(u_n^{(m)})]\in\ker(\wt{N^+} - \ol\la)$ such that $[(x_n^{(m)})] - (\wt N  - \la)[(u_n^{(m)})]\to 0$ as $m\to\infty$ in $\wt\calH$. Hence, the following holds:
\begin{enumerate}
\item[(a)] $\lim_{m\to\infty}\,\limsup_{n\to\infty}\,\|x_n^{(m)} - (N - \la)u_n^{(m)}\| = 0$,
\item[(b)] $\forall m\in\N : \lim_{n\to\infty}\,\|(N^+ - \ol\la)u_n^{(m)}\| = 0$,
\item[(c)] $\forall m\in\N : \limsup_{n\to\infty}\,\|x_n^{(m)}\| = 1$,
\item[(d)] $\lim_{m\to\infty}\,\limsup_{n\to\infty}\,\|(N - \mu)x_n^{(m)}\| = 0$.
\end{enumerate}
It is not difficult to see that from (a)--(d) it follows that for each $k\in\N$ there exist $m_k,n_k\in\N$ such that
\begin{enumerate}
\item[(a')] $\|x_{n_k}^{(m_k)} - (N - \la)u_{n_k}^{(m_k)}\| < \frac 1 k$,
\item[(b')] $\|(N^+ - \ol\la)u_{n_k}^{(m_k)}\| < \frac 1 k$,
\item[(c')] $\frac 1 2\le\|x_{n_k}^{(m_k)}\|\le 2$,
\item[(d')] $\|(N - \mu)x_{n_k}^{(m_k)}\| < \frac 1 k$.
\end{enumerate}
Set $x_k := x_{n_k}^{(m_k)}$ and $u_k := u_{n_k}^{(m_k)}$. From (c') and (d') we conclude $\mu\in\sap(N)$ and hence $\mu\in\sp(N)$. Consequently,
$$
\liminf_{k\to\infty}\,[x_k,x_k] > 0.
$$
On the other hand, we have
\begin{align*}
\big|[x_k,x_k]\big|
&\le \big|[x_k - (N - \la)u_k,x_k]\big| + \big|[(N - \la)u_k,x_k - (N - \la)u_k]\big|\\
&\qquad\quad + \big|[(N - \la)u_k,(N - \la)u_k]\big|\\
&\le \frac 2 k + \frac 1 k\,\|(N - \la)u_k\| + \big|[(N^+ - \ol\la)u_k,(N^+ - \ol\la)u_k]\big|\\
&\le \frac 2 k + \frac 1 k\left(\frac 1 k + \|x_k\|\right) + \frac 1 {k^2}\le\frac 6 k\,,
\end{align*}
which is a contradiction.
\end{proof}

\begin{rem}
Note that in Theorem \ref{t:spectral_set} it is only assumed that $\sigma$ is of positive type and not of two-sided positive type with respect to $N$.
\end{rem}

In what follows we derive some direct consequences of Theorem \ref{t:spectral_set} (see also Lemma \ref{l:basic2} and Corollary \ref{c:res}).

\begin{cor}\label{c:1st}
If $\sigma$ is a spectral set of $N$ which is of positive type with respect to $N$, then $\sigma$ is of two-sided positive type with respect to $N$. In particular, if $\la\in\sp(N)$ is an isolated point of $\sigma(N)$, then $\la\in\spp(N)$, and $\la$ is a pole of order one of the resolvent of $N$.
\end{cor}

\begin{cor}\label{c:finite_dim}
If $\dim\calH < \infty$, then $\sp(N) = \spp(N)$. In particular, if for some $\la\in\sigma(N)$ the inner product $\product$ is positive definite on $\ker(N - \la)$ or on $\ker(N^+ - \ol\la)$, then
$$
\ker(N - \la) = \ker(N^+ - \ol\la) = \calL_\la(N) = \calL_{\ol\la}(N^+).
$$
\end{cor}

\begin{cor}
Let $\calS\subset\C$ be an open set and assume that $N$ has a local spectral function $E$ on $\calS$. Then $\calS$ is of positive type with respect to $N$ if and only if $E$ is a local spectral function of positive type.
\end{cor}
\begin{proof}
If $E$ is of positive type, then $\calS$ is of positive type with respect to $N$ by Lemma \ref{l:lsf_nec}. Conversely, assume that $\calS$ is of positive type with respect to $N$. Let $\Delta\in\mathfrak B(\calS)$. Then from Lemma \ref{l:proj} we conclude that $Q := E(\Delta)$ is selfadjoint. It remains to show that $(Q\calH,\product)$ is a Hilbert space. As $N^+$ commutes with $Q$, it follows that $N|Q\calH$ is a normal operator in the Krein space $(Q\calH,\product)$ with $(N|Q\calH)^+ = N^+|Q\calH$. The assertion is now a consequence of $\sap(N|Q\calH)\subset\sp(N|Q\calH)$ and Theorem \ref{t:spectral_set}.
\end{proof}

A bounded operator $T$ in $(\calH,\product)$ is said to be {\it fundamentally reducible} if there exists a fundamental decomposition $\calH = \calH_+[\ds]\calH_-$ of $\calH$ such that both $\calH_+$ and $\calH_-$ are $T$-invariant. Note that a fundamentally reducible {\it normal} operator is always normal in a Hilbert space. A fundamentally reducible operator $T$ is called {\it strongly stable} if
$$
\sigma(T|\calH_+)\cap\sigma(T|\calH_-) = \emptyset,
$$
cf.\ \cite{b98}. The following corollary was already proved in \cite{pst} under the additional assumption that $\sigma(\Im N)\subset\R$ and that a growth condition \eqref{e:growth} on the resolvent of $\Im N$ holds near $\R$. Here, it immediately follows from \cite[Theorem 1]{b98} and Theorem \ref{t:spectral_set}.

\begin{cor}\label{c:strongly_stable}
The following statements are equivalent.
\begin{enumerate}
\item[{\rm (i)}]   $N$ is strongly stable.
\item[{\rm (ii)}]  There exists $\delta > 0$ such that every normal operator $X$ with $\|X - N\| < \delta$ is fundamentally reducible.
\item[{\rm (iii)}] $\sigma(N) = \sp(N)\cup\sm(N)$.
\end{enumerate}
\end{cor}

In view of Corollary \ref{c:1st} the question arises whether the sets $\sp(N)$ and $\spp(N)$ possibly even coincide. We cannot give a definite answer to this question here. However, the following proposition shows that a possible counterexample can only be found in an infinite-dimensional Krein space which is not a Pontryagin space.

\begin{prop}\label{p:pont}
If $(\calH,\product)$ is a Pontryagin space, then we have $\sp(N) = \spp(N)$.
\end{prop}
\begin{proof}
It is easy to see (see also \cite{ls}) that the Pontryagin space $(\calH,\product)$ can be decomposed into a direct orthogonal sum $\calH = \calH_1\,[\ds]\,\calH_2$ with closed $N$- and $N^+$-invariant subspaces $\calH_1$ and $\calH_2$ such that $\dim\calH_1 < \infty$ and the operators $\Re N|\calH_2$ and $\Im N|\calH_2$ have real spectra. Set $N_j := N|\calH_j$, $j = 1,2$. Owing to the properties of selfadjoint operators in Pontryagin spaces and \cite[Theorem 4.2]{pst} the operator $N_2$ has a local spectral function of positive type on neighborhoods of spectral points of positive type of $N_2$. The assertion now follows from Lemma \ref{l:lsf_nec} and Corollary \ref{c:finite_dim}.
\end{proof}

\section{Concluding remarks}
A selfadjoint operator in a Krein space has a local spectral function of positive type on sets which are of positive type. We proved that the same holds for normal operators if and only if the set is of two-sided positive type. The question whether the notions "positive type" and "two-sided positive type" are in fact equivalent has to be left open, but the answer is "yes" if the set is a spectral set or if the Krein space is a Pontryagin space.

\section*{Acknowledgements}
The author gratefully acknowledges the support of the Deutsche For\-schungs\-gemeinschaft (DFG) under grant BE 3765/5-1.


\end{document}